\documentclass[a4paper,11pt]{amsart}

\usepackage{anysize} \marginsize{0.7in}{0.7in}{1in}{1in}
\usepackage[a4paper, margin=3.5cm]{geometry}
\usepackage{comment}
\usepackage{xcolor}
\usepackage{amsmath}
\usepackage{mathtools}
\usepackage[all]{xy}
\usepackage[utf8]{inputenc}
\usepackage{varioref}
\usepackage[normalem]{ulem}
\usepackage{amsfonts}
\usepackage{amssymb}
\usepackage{bbm}
\usepackage{esint}
\usepackage{graphicx}
\usepackage{tikz}
\usepackage{empheq}
\usepackage{enumitem}
\usepackage{tikz-cd}
\usetikzlibrary{matrix,arrows,decorations.pathmorphing}
\usepackage{mathrsfs}
\usepackage{ifthen}
  \usepackage[hypertexnames=false,backref=page,pdftex,
 	pdfpagemode=UseNone,
 	breaklinks=true,
 	extension=pdf,
 	colorlinks=true,
 	linkcolor=blue,
 	citecolor=red,
 	urlcolor=blue,
 ]{hyperref}

\usepackage{cleveref}  

\makeatletter
\newcommand{\skipitems}[1]{%
  \addtocounter{\@enumctr}{#1}%
}
\makeatother

\usepackage{soul}

\usepackage{geometry}

\definecolor{brickred}{rgb}{0.8, 0.25, 0.33}

\usepackage[textsize=small]{todonotes}

\newcommand{\Q}{{\mathbb Q}}

\newcommand{\Vz}{V_\zeta}
\newcommand{\Vb}{{V_{\overline{\zeta}}}}

\newcommand{\eig}{{\mathrm{eig}}}
\newcommand{\Eig}{{\mathrm{Eig}}}

\newcommand{\bbQ}{{\mathbb Q}}

\newcommand{\V}[2]{V_{(#1,#2)}}
\newcommand{\Ver}{V_{(5)}}
\newcommand{\mH}{{\mathrm H}}
\newcommand{\kg}{{\mathfrak g}}

\newcommand{\Aut}{\operatorname{Aut}}

\newcommand{\ord}{{\mathrm{ord}}}

\DeclareMathOperator{\Sym}{Sym}

\theoremstyle{plain}
\newtheorem{satz}{Satz}[section]
\newtheorem{theorem}[satz]{Theorem}
\newtheorem{definition}[satz]{Definition}

\newtheorem{lemma}[satz]{Lemma}
\newtheorem{corollary}[satz]{Corollary}

\newtheorem{proposition}[satz]{Proposition}

\theoremstyle{remark}

\newtheorem{remark}[subsection]{Remark}

\makeatletter
\@namedef{subjclassname@2020}{%
 \textup{2020} Mathematics Subject Classification}
\makeatother



\thanks{}

\title{Non-existence of Enriques manifolds from OG10 type manifolds}

\author[S.Billi]{Simone Billi}
\address[Simone Billi]{Università di Genova, Dipartimento di Matematica, Via Dodecaneso, 35, 16146 Genova, Italy}
\email{simone.billi@edu.unige.it}

\author[F.Giovenzana]{Franco Giovenzana}
\address[Franco Giovenzana]{Université Paris-Saclay, CNRS, Laboratoire de Mathématiques d'Orsay, Rue Michel Magat, Bât. 307, 91405 Orsay, France}
\email{franco.giovenzana@universite-paris-saclay.fr}

\author[L. Giovenzana]{Luca Giovenzana}
\address[Luca Giovenzana]{Department of Pure Mathematics\\ University of Sheffield\\ Hicks Building, Hounsfield Road\\ Sheffield, S3 7RH\\ UK}
\email{l.giovenzana@sheffield.ac.uk}

\author[A.Grossi]{Annalisa Grossi}
\address[Annalisa Grossi]{Alma Mater studiorum Università di Bologna Dipartimento di Matematica,
Piazza di Porta San Donato 5, Bologna, 40126 Italia}
\email{annalisa.grossi3@unibo.it}

\thanks{\tiny{S.B. was partially supported by the Curiosity Driven 2021 Project "Varieties with trivial or
negative canonical bundle and the birational geometry of moduli spaces of curves: a
constructive approach" - Programma nazionale per la Ricerca (PNR) DM 737/2021.\\
F.G. was funded by Deutsche Forschungsgemeinschaft
(DFG, German Research Foundation), Projektnummer 509501007, and partially supported by the European Research Council (ERC) under the
European Union’s Horizon 2020 research and innovation programme (ERC-2020-SyG-854361-
HyperK).\\
A.G. was partially supported by the European Research Council (ERC) under the European Union’s Horizon 2020 research and innovation programme (ERC-2020-SyG-854361-HyperK).\\
S.B. and A.G. were partially supported by the European Union - NextGenerationEU under the National Recovery and Resilience Plan (PNRR) - Mission 4 Education and research - Component 2 From research to business - Investment 1.1 Notice Prin 2022 - DD N. 104 del 2/2/2022, from title \lq\lq Symplectic varieties: their interplay with Fano manifolds and derived categories\rq\rq, proposal code 2022PEKYBJ – CUP J53D23003840006.\\
All authors are members of the INdAM group GNSAGA, and S.B. was partially supported by GNSAGA}}

\usepackage{framed}

\begin{document}
\begin{abstract}
We use the LLV algebra to describe the action of a finite order automorphism on the total cohomology of a manifold of OG10 type.
As an application, we prove that no Enriques manifolds arise as étale quotients of hyper-Kähler manifolds of OG10 type. This answers a question raised in \cite{Pacienza_Sarti}.
\end{abstract}

\maketitle

\section{Introduction}

The celebrated Beauville–Bogomolov decomposition theorem states that any compact complex Kähler manifold with a numerically trivial canonical bundle decomposes, after a finite étale morphism, into a product of three types of varieties: hyper-Kähler manifolds, strict Calabi–Yau manifolds, and abelian varieties. These three building blocks have been the subject of extensive study over the past decades, with hyper-Kähler manifolds (discovered in the 1980s) attracting particular attention due to their rich geometry and connection to various fields of mathematics.

Recently, attention has shifted towards understanding complex Kähler manifolds with numerically trivial canonical bundle in greater generality. In this context, Enriques manifolds have been introduced independently by \cite{BN-WS} and \cite{OS} with slightly different definitions.
For the result in the current paper, we adopt the definition given in \cite{OS}. This excludes the case where the universal cover is Calabi-Yau, which was allowed in the definition of \cite{Boissiere.NieperWisskirchen.Sarti:enriques.varieties}.
\begin{definition}\label{def_enr_man}
An Enriques manifold is a connected complex manifold that is not simply connected and whose universal cover is a hyper-Kähler manifold.
\end{definition}
Various results related to these manifolds have been established: 
the Torelli Theorem \cite{OS_periods, raman2024}, the termination of flips \cite{denisi2024mmp}, and the Kawamata-Morrison cone conjecture for all Enriques manifolds of prime index \cite{Pacienza_Sarti}. Also, singular counterparts have been introduced and analyzed \cite{boissiere2024logarithmic, denisi2024mmp}.

Unlike the two-dimensional case, where Enriques surfaces are well-understood and have a degree-2 universal cover, higher dimensional Enriques manifolds exhibit greater complexity. Examples of index \(2\), \(3\) and \(4\) are constructed starting from a hyper-Kähler manifold of K3\(^{[n]}\)-type or of \(K_n(A)\)-type \cite{BN-WS,OS}. 
 Nevertheless, to date, these are all the degrees for which examples of Enriques manifolds are known. It is then crucial to produce more examples of Enriques manifolds or to demonstrate that certain deformation types of hyper-Kähler manifolds cannot cover Enriques manifolds.

In this paper, we address the latter goal by proving that no Enriques manifold can arise as an étale quotient of a hyper-Kähler manifold deformation equivalent to O’Grady’s ten-dimensional example, answering a question of \cite{Pacienza_Sarti}.

\subsection{Results}
The work of Verbitsky \cite{Ver90,Ver95,Ver96} and Looijenga--Lunts \cite{LL97} showed that the cohomology $H^*(X,\mathbb Q)$ of a hyper-Kähler manifold $X$ admits a natural action of the Lie algebra $\mathfrak g=\mathfrak{so}(4,b_2(X)-2)$ generalizing the action of $\mathfrak sl(2)$ given by the hard Lefschetz Theorem. The algebra $\mathfrak g$ can be identified with the special orthogonal algebra of the Mukai completion $V=\mH^2(X,\bbQ)\oplus U$ where $U$ is the hyperbolic plane.

The decomposition in $\kg$-representations for the known deformation families of hyper-Kähler manifolds is computed in \cite{GKLR}. If $X$ is a hyper-Kähler manifold of OG10 type, then
\begin{equation}\label{tot-LLV-dec}
    H^*(X, \mathbb Q) \simeq V_{(5)}\oplus V_{(2,2)}
    \end{equation}
    where $V_{(5)}$ is the Verbitsky component, which is isomorphic to the largest submodule of $\Sym^5 V$, and $V_{(2,2)}$ is the largest irreducible submodule of $ \Sym^2(\Lambda^2V)$.

 Given a finite order automorphism $\varphi$ of \(X\) we define an isometry $\gamma$ on the Mukai completion $V = H^2(X)\oplus U$ by setting
 \begin{align*}
     \gamma|_{H^2} = \varphi^*|_{H^2}\qquad \mbox{ and  }\qquad \gamma|_U = \mathrm{id}_U.
 \end{align*}
Any such $\gamma$ naturally induces an automorphism $\gamma_{\rho}$ on any $\kg$--representation $V_\rho$.

As a first result we determine, with an ambiguity for an even order automorphism, the dimension of the $\varphi^*$-invariant part of $H^*(X,\mathbb Q)$ in terms of the invariant part on $H^2(X,\mathbb Q)$.
If $\varphi$ is a linear automorphism  on a vector space $W$, we denote by $\Eig(\varphi, W, \lambda)$ the eigenspace relative to the eigenvalue $\lambda$ and by $\eig(\varphi, W, \lambda)$ its dimension. 
\begin{theorem}\label{MainThm}
Let \(X\) be an irreducible holomorphic symplectic manifold of OG10 type, and let \(\varphi \in \Aut(X)\) be an automorphism of finite order, then the following 
relations hold.
\begin{itemize}
    \item If \(\ord(\varphi)\) is even, then one of the following equalities hold true:

\begin{enumerate}[label=\alph*)]
  \item \(\eig(\varphi^*, H^*(X, \mathbb{Q}), 1) = \eig(\gamma_{(5)},\Ver, 1) + \eig(\gamma_{(2,2)}, \V22, 1)\);
 \item \(\eig(\varphi^*, H^*(X, \mathbb{Q}), 1) = \eig(\gamma_{(5)}, \Ver, 1) - \eig(\gamma_{(2,2)}, \V22, 1)\).
    \end{enumerate}

\item If \(\ord(\varphi)\) is odd, then:

 \begin{enumerate}[label=\alph*)]
         \skipitems{2}
        \item \(\eig(\varphi^*, H^*(X, \mathbb{Q}), 1) = \eig(\gamma_{(5)}, \Ver, 1) + \eig(\gamma_{(2,2)}, \V22, 1)\).
    \end{enumerate}
\end{itemize}
Moreover, if $\ord (\varphi)\in \{2,3\}$, then \(\eig(\varphi^*, H^*(X, \mathbb{Q}), 1)\) is a polynomial in the variable $\eig(\varphi^*, H^2(X, \mathbb{Q}), 1)$ appearing in Table~\ref{table_pols}.

\end{theorem}
A straightforward application of Theorem \ref{MainThm} gives the main result of the paper.
\begin{corollary}\label{cor:no-Enriques}
    There is no Enriques manifold whose universal covering space is a hyper-Kähler manifold of OG10 type.
\end{corollary}

\subsection*{Acknowledgments}
 The authors would like to thank Emanuele Macrì and Jieao Song for useful hints and discussion. AG would like to thank Alessandra Sarti for asking her about the existence of other examples of Enriques manifolds. FG thanks Giuseppe Ancona and Chenyu Bai for convincing him not to pursue an incorrect approach to the problem.
 We would like to thank Giovanni Mongardi, Gianluca Pacienza, Alessandra Sarti and Francesco Denisi for reading a preliminary version of the paper.
\section{Preliminaries}
In this section, we recall basic facts on Enriques manifolds and on the LLV-algebra.
\subsection{Enriques manifolds}
The definition that we refer to is Definition \ref{def_enr_man}. For the sake of clarification, it coincides with the definition of \textit{weak Enriques varieties} in \cite[\S2.2]{BN-WS} with the additional assumption of being irreducible with universal cover not strict Calabi-Yau, leaving the only possibility for the universal cover to be hyper-Kähler. Consider a compact hyper-Kähler manifold \(X\) of dimension \(2n\), \(n \geq 2\), and let \(f\) be an automorphism of \(X\) of order \(d \geq 2 \) such that the cyclic group generated by \(\langle f \rangle \) acts freely on \(X\).
It follows that \(f\) acts purely non-symplectic, i.e., there exists a primitive \(d\)-th root of the unity \(\zeta\), with \(d \geq 2\), such that the action of \(f\) on the symplectic form \(\sigma_{X}\) is given by \(f^*(\sigma_X)=\zeta \sigma_X\).
By a result of Beauville (\cite[Proposition 6]{Beauville:some.remarks}) the variety \(X\) is necessarily projective.
The étale quotient \(Y=X/\langle f \rangle\) is a smooth projective Enriques manifold of index \(d\).

\subsection{The LLV algebra}
If \(\omega\in H^2(X,\bbQ)\) is a Kähler class, there is an associated Lefschetz operator \(L_\omega=\omega\cup-\). It forms an \(\mathfrak{sl}(2)\)-triple \(\{L_\omega,h,\Lambda_\omega\}\) where \(\Lambda_\omega\) is the dual Lefschetz operator, and \(h\colon \mH^*(X,\bbQ) \to \mH^*(X,\bbQ)\) is the degree operator $h\cdot x= (k-\dim X)x$ for \(x\in \mH^k(X, \bbQ)\).
For any class \(\alpha\in H^2(X,\bbQ)\), the operator \(L_\alpha\) is still well defined and the existence (and uniqueness) of an operator \(\Lambda_\alpha\) satisfying the relation \([L_\alpha,\Lambda_\alpha]=h\) is an open condition in \(H^2(X,\bbQ)\) by the Jacobson–Morozov Theorem.

\begin{definition}
    The \textit{Looijenga–Lunts–Verbitsky (LLV) algebra} \(\mathfrak{g}\) of \(X\) is the Lie subalgebra of \(\mathfrak{gl}(H^*(X,\bbQ))\) generated by all the \(\mathfrak{sl}(2)\)-triples associated with classes \(\alpha\in H^2(X,\bbQ)\).
\end{definition}
The Lie algebra \(\mathfrak{g}\) is a semisimple algebra defined over \(\bbQ\).
Work of Verbitsky \cite{Ver90} and Looijenga-Lunts \cite{LL97} shows that $\kg$ can be identified with the special orthogonal algebra of the Mukai completion $V=\mH^2(X,\bbQ)\oplus U$ where $U$ is the hyperbolic plane.  In particular, it is isomorphic to the Lie algebra \(\mathfrak{so}(4,b_2(X)-2)\).

A remarkable fact is that the Lie algebra \(\mathfrak{g}\) and the decomposition of \(H^*(X,\bbQ)\) in irreducible \(\mathfrak{g}\)-representations are diffeomorphism invariants of \(X\), and hence they depend only on the deformation class of \(X\). The irreducible \(\mathfrak{g}\)-representations are indexed by their highest weights, which are
non-negative integral linear combinations of the fundamental weights. An irreducible \(\mathfrak{g}\)-representation with highest weight \(\mu\) will be denoted by \(V_\mu\).  

If \(X\) is a manifold of OG10 type, then the LLV algebra \(\mathfrak{g}\) is a Lie algebra of type D. The decomposition 
\[H^*(X,\bbQ)\cong V_{(5)}\oplus V_{(2,2)}\]
in irreducible \(\mathfrak{g}\)-representations has been computed in \cite[Theorem 1.2 (iv)]{GKLR}. We refer to \cite[Appendix A.2]{GKLR} for a detailed description of the weights. Here the representation \(V_{(5)}\) is naturally identified with the image of the map \(\Sym^* H^2(X,\bbQ)\to H^*(X,\bbQ)\), and it is called the \textit{Verbitsky component}, according to the original paper \cite{Ver90}.

\section{Proofs}
In this section we give the complete proofs of the main results, namely Theorem~\ref{MainThm} and Corollary~\ref{cor:no-Enriques}.
\subsection*{Notation}
Since we only consider cohomology groups of OG10 type manifolds with rational coefficients, we simplify the notation by writing $H^*$ for the full cohomology $H^*(X, \Q)$ and similarly by $H^2$ for the cohomology group $H^2(X, \mathbb Q)$. If \(V\) is a \(\mathbb{Q}\)-vector space, then we denote by \(V_{\mathbb{C}}\) the space \(V\otimes\mathbb{C}\).

For an automorphism $\gamma$ on a vector space \(V\), we denote the eigenspace corresponding to the eigenvalue $\lambda\in\mathbb{C}$ by $\Eig(\gamma,V, \lambda)$, omitting the explicit mention of $\gamma$ when there is no risk of confusion.

\begin{remark}
    Given a Lie algebra $\kg$ acting on a vector space $V$, we observe that the smallest $\kg$-submodule containing a subset $W\subset V$ is given by
    \begin{align*}
        \kg \cdot W = \left\{ \sum g_{i1}\cdot g_{i2} \cdots g_{ij_i}\cdot x_i \mbox{ such that } x_i\in W,\ g_{ji}\in\kg, \mbox{ and the sum is finite}\right\}.
    \end{align*}
\end{remark}

    \begin{proposition}\label{prop:AB-decomposition}
        There exist unique $\mathfrak g$-submodules $A,B \subset H^*(X)$ such that $A \simeq V_{(5)}$ and $B \simeq V_{(2,2)}$, so that 
        \begin{align}\label{decomposition}
            H^*(X,\mathbb{Q}) = A \oplus B
        \end{align}
        as $\kg$-representations. 
    \end{proposition}
    \begin{proof}
        From the original paper of Verbitsky \cite{Ver90} the submodule $\Ver$ is naturally identified with the image of the natural morphism $\Sym^* H^2(X,\Q) \to H^*(X,\Q)$, which can be described also as $\mathfrak g \cdot H^0(X,\Q)$. Let $B, B'$ be two $\mathfrak g$-submodules of $H^*(X,\Q)$ isomorphic to $\V22$. Then the composition 
        \[
        B\hookrightarrow H^*(X,\Q) \to H^*(X,\Q)/B' \simeq \Ver
        \]
        is the zero morphism by Schur's lemma. Hence $B$ coincides with the kernel of the projection, that is $B'$.
    \end{proof}
 
    \begin{proposition}\label{prop:AB-eig}
        Let $\varphi\in \Aut (X)$ be of finite order. Then $\varphi^*$ preserves the LLV decomposition \eqref{decomposition}, i.e. $\varphi^*(A) = A$ and $\varphi^*(B) = B$, so that
        \[
        \eig (\varphi^*, H^*(X,\Q),1) = \eig (\varphi^*, A,1) + \eig(\varphi^*, B,1).
        \]
    \end{proposition}
    \begin{proof}
        As $A= \mathfrak g \cdot H^0$ we can write any element $x\in A$ as
        \[
        x = \beta + \sum \alpha_{i_1} \cup \cdots\cup \alpha_{i_r}
        \]
        for some $\beta\in H^0$ and $\alpha_{i_j}\in H^2$, so that we compute 
        \[
        \varphi^*x = \varphi^*\beta + \sum \varphi^*\alpha_{i_1} \cup \cdots\cup \varphi^*\alpha_{i_r}
        \]
        and conclude that $\varphi^* x\in A$, as wanted.

        Now we turn our attention to the $\kg$-submodule $B\simeq\V22$. First, notice that the direct sum
        \[
        H^{10}(X,\Q) = \left( A \cap H^{10}(X,\Q) \right) \oplus \left( B \cap H^{10}(X,\Q)\right)
        \]
        is orthogonal with respect to the intersection form. Indeed, the intersection product of an element $a\in A\cap H^{10}$ with an element $b\in B\cap H^{10}$ can be expressed in terms of the $\kg$-action on $b$, and \(B\cap H^{20}=0\).
        As $\varphi^*$ acts as an isometry for the intersection form and $\varphi^* (A) = A$, for any $x\in B \cap H^{10}$ we have $\varphi^* x\in B \cap H^{10}$. As $B\simeq\V22$ is irreducible, we can write $B = \mathfrak g \cdot x$ for any nonzero element $x\in B \cap H^{10}$, so that any element $y\in B$ can be written as finite sum of elements of the form:
\begin{align*}
    g_1\cdot g_2\cdots g_r \cdot x
\end{align*}
where $g_i=L_{\alpha}$ or $g_i=\Lambda_{\beta}$ for some $\alpha, \beta \in H^2(X,\mathbb Q)$.
        We compute
        \[
        \varphi^* (L_\alpha\cdot x) = \varphi^* (\alpha\cup x) = \varphi^* \alpha\cup \varphi^* x = L_{\varphi^* \alpha}\cdot \varphi^* x
        \] 
        and similarly for the dual Lefschetz operator $\Lambda_\beta := *^{-1}\circ L_\beta \circ *$ for some $\beta\in H^2$, i.e. $\varphi^*(\Lambda_\beta \cdot x) = \Lambda_{\varphi^*\beta} \cdot \varphi^* x$. We conclude that $\varphi^* y\in B$ as it is the finite sum of 
\begin{align*}
    \varphi^* (g_1\cdot g_2\cdots g_r \cdot x ) = g_1'\cdot g_2'\cdots g_r' \cdot \varphi^*x
\end{align*}
where $g_i'\in \kg$ and $\varphi^*x \in B$ as discussed above.
    \end{proof}
We now discuss the interaction between the $\kg$-action and the action on $H^*(X)$ of an automorphism $\varphi$.
Let $V = H^2(X,\Q) \oplus U$ be the Mukai extension. Let us define an automorphism $\gamma$ on $V$ by setting 
\begin{align}\label{eq:gamma}
    \gamma|_{H^2} = \varphi^*|_{H^2},\qquad \gamma|_U = \mathrm{id}_U.
\end{align}
This induces an automorphism $\gamma_\rho$ on any $\kg$-representation $V_\rho$. 

\begin{proposition}\label{prop: B-V22-compatibility}
    Let $\varphi\in \Aut (X)$ be of finite order $\ord (\varphi)=n$.
    
    If $n$ is even, then
    \begin{align*}
    \text{either }  &\eig(\varphi^*, B, 1) = \eig (\gamma_{(2,2)}, \V22, 1) \\
    \text{or } &\eig(\varphi^*, B, 1) = \eig (-\gamma_{(2,2)}, \V22, 1).
    \end{align*}
    
    If $n$ is odd, then
    \[
    \eig(\varphi^*, B, 1) = \eig (\gamma_{(2,2)}, \V22, 1).
    \]
\end{proposition}

\begin{proof}
    Let $\varphi\in \Aut(X)$ be any automorphism of order $n$. Let $f$ be any isomorphism $B\xrightarrow{\sim} \V22$ of $\kg$-representations. Then $\psi := f\circ \varphi^*\circ f^{-1}$ is an automorphism of $\V22$ which satisfies
    \[
    \psi (L_\alpha \cdot x) = L_{\varphi^*\alpha} \cdot \psi (x)
    \]
    and $f|\colon \Eig(\varphi, B,1) \xrightarrow{\sim} \Eig(\psi,\V22,1)$.
    Let $\gamma$ be defined as in Equation~\eqref{eq:gamma}, and let $\gamma_{(2,2)}$ be the induced automorphism on $\V22$. We have
    \[
    \gamma_{(2,2)} (L_\alpha \cdot x) = L_{\varphi^*\alpha} \cdot \gamma_{(2,2)} (x).
    \]

One has that $\gamma_{(2,2)}^n = \mathrm{id}$, and the composition $\psi\circ \gamma^{n-1}$ is an automorphism of $\kg$-representations, indeed
\[
\psi\circ\gamma_{(2,2)}^{n-1} (L_\alpha \cdot x) =
\psi(L_{(\varphi^*)^{n-1}\alpha} \cdot \gamma_{(2,2)}^{n-1}(x)) =
L_{\varphi^*(\varphi^*)^{n-1}\alpha} \cdot \psi \circ\gamma_{(2,2)}^{n-1}(x) =
L_\alpha \cdot \psi \circ\gamma_{(2,2)}^{n-1}(x).
\]

By Schur's lemma every $\kg$-isomorphism of $\V22$ differs by a non-trivial constant, in particular  $\psi\circ \gamma_{(2,2)}^{n-1} = \psi \circ \gamma_{(2,2)}^{-1}=\lambda\cdot\mathrm{id}_{\V22}$ for some $\lambda\in \mathbb Q^*$, so that $\psi = \lambda \cdot \gamma_{(2,2)}$. By the equality 
\begin{align*}
\lambda^n\cdot\mathrm{id} = (\lambda\cdot \gamma_{(2,2)})^n   = \psi^n = \mathrm{id}
\end{align*}
we get that $\lambda =1$ if $n$ is odd, and $\lambda= 1$ or $\lambda=-1$ if $n$ is even. This shows the claim.
\end{proof}

\begin{lemma}\label{lem:V22-compuation}
If $\ord(\varphi) = 2$ with $r:=\eig(\varphi^*, H^2,1)$, then
    \begin{align}\label{eq:ord2-V22}
    \eig (\gamma_{(2,2)}, \V22, 1) =  \frac{2}{3}r^4 - \frac{88}{3}r^3 + \frac{1447}{3}r^2 + \frac{10538}{3}r + 28500.
    \end{align}
    
If $\ord(\varphi) = 3$ with $2r:=\eig(\varphi^*,H^2,1)$, then
    \begin{align}\label{eq:ord3-V22}
        \eig (\gamma_{(2,2)}, \V22, 1) =\frac{9}{2}r^4 - 60r^3 + \frac{597}{2}r^2 - 675r + 13158.
    \end{align}
\end{lemma}
\begin{proof}
A computation in representation theory
 gives the following isomorphism of $\mathfrak g$-representations
        \[
        \Sym^2 \Lambda^2 V \simeq \Sym^2 V \oplus \Lambda^4 V \oplus \V22.
        \] 
This can be readily checked on Sage \cite{sage} with the following lines\footnote{We thank Jieao Song for explaining this to us.}
\begin{verbatim}
D=WeylCharacterRing("D13");
V.exterior_power(2).symmetric_power(2)
V.exterior_power(4)
V.symmetric_power(2)
\end{verbatim}
We firstly treat the case of an involution $\gamma$ on $V$, so that $\eig(\gamma,V,1) = r+2$ where $r:= \eig(\varphi^*,H^2,1)$. We define $V_+ = \Eig(V,1)$ and $V_- = \Eig(V,-1)$ and we compute
\[
\Sym^2V = \Sym^2(V_+\oplus V_-) = \underbrace{\Sym^2 V_+ \oplus \Sym^2 V_-}_{\Eig(1)} \oplus \underbrace{V_+\otimes V_-}_{\Eig(-1)}
\]
thus 
\begin{align}\label{formula_2}
\eig(\Sym^2 V,1) = 
r^2-22r+303.
\end{align}

Analagously, one computes
\begin{align*}
    (\Lambda^4 V)_+ &= \Lambda^4 V_+ \oplus (\Lambda^2 V_+ \otimes \Lambda^2 V_-) \oplus \Lambda^4 V_-
    \end{align*}
    from which it follows that 
    \begin{align*}
    \eig(\Lambda^4 V,1) &= 
    \frac{1}{3}r^4 - \frac{44}{3}r^3 + \frac{689}{3}r^2 - \frac{4510}{3}r + 10902.
\end{align*}
Similarly as before, we have 
\begin{align*}
(\Lambda^2V)_+ = \Lambda^2V_+ \oplus \Lambda^2V_-
\end{align*}
which has dimension
\begin{align*}
\eig(\Lambda^2 V,1)=
r^2 - 22r + 277.
\end{align*}
One can easily check that 
\begin{align*} \Eig(\Sym^2\Lambda^2 V,1) &= \Sym^2( (\Lambda^2 V)_+) \oplus\Sym^2( (\Lambda^2 V)_-)
\end{align*}
and compute 
\begin{align*}
\eig(\Sym^2 \Lambda^2 V,1) &=
r^4 - 44r^3 + 713r^2 + 5038r + 39679.
\end{align*}
Finally, putting all together we get
\begin{align*}
\eig(\V22,1) &= \eig(\Sym^2 \Lambda^2 V,1) - \eig(\Lambda^4 V,1) - \eig(\Sym^2 V,1) =\\
&= \frac{2}{3}r^4 - \frac{88}{3}r^3 + \frac{1447}{3}r^2 + \frac{10538}{3}r + 28500.
\end{align*}
\bigskip

We now treat the case of order 3. We consider the decomposition \[V  \otimes \mathbb C = V_+ \oplus \Vz \oplus \Vb\] in eigenspaces relative to \(1,\zeta,\overline{\zeta}\) respectively, where \(\zeta\) is a primitive third root of unity. Notice that as the action of $\varphi^*$ on $H^2(X)$ is integral, we have $\eig(\varphi^*,H^2_\mathbb C,\zeta) = \eig(\varphi^*,H^2_\mathbb C, \overline{\zeta})$. Thus, $\dim H^2(X,\Q)^{\varphi^*}$ is even and we denote it by $2r$, so that $\eig(\varphi^*, H^2(X,\mathbb C),\zeta) = \eig(\varphi^*,  H^2(X,\mathbb C),\overline\zeta) = 12-r$ and $\eig(\varphi^*, V_\mathbb C, \overline\zeta) = 12-r$. 
This gives the following decomposition
{\small\begin{align*}
    \Eig(\Lambda^4 V,1)=     \Lambda^4 V_+ \oplus \big(\Lambda^2 V_+ \otimes (\Vz \otimes \Vb)\big) \oplus \big(\Lambda^1V_+ \otimes (\Lambda^3\Vz \oplus \Lambda^3\Vb)\big) \oplus \big((\Lambda^2\Vz \otimes \Lambda^2\Vb)\big),
\end{align*}}
 from which we get 
\begin{align*}
    \eig(\Lambda^4 V,1)=
    \frac{9}{4}r^4-\frac{69}{2}r^3+\frac{783}{4}r^2-\frac{943}{2}r+5380.
\end{align*}
 Moreover, we have 
\begin{align*}
    \Eig(\Sym^2 V,1) &= \Sym^2 V_+\oplus \big(\Vz \otimes \Vb\big).
\end{align*}
and, since \(\dim V_+=2r+2\), it immediately follows that
\begin{align*}
    \eig(\Sym^2V,1)=
    3r^2-19r+147.
\end{align*}
Similarly, we have 
\begin{align*}
   \Eig(\Lambda^2 V,1) &=  \Lambda^2 V_+ \oplus \big(\Vz \otimes \Vb\big)
\end{align*}
and from the relation 
\begin{align*}\label{formula_3}
    \eig(\Lambda^2 V,\zeta)=
    \frac{\dim(\Lambda^2V)-\eig(\Lambda^2 V,1)}{2},
\end{align*}
we get 
\begin{align*}
    \eig(\Lambda^2 V,1)=3r^2-21r+145
\end{align*}
and 
\begin{align*}
    \eig(\Lambda^2 V,\zeta)=-\frac{3}{2}r^2+\frac{21}{3}r+90.
\end{align*}
Letting \((\Lambda^2 V)_+\), \((\Lambda^2 V)_\zeta\) and \((\Lambda^2 V)_{\overline{\zeta}}\) be the corresponding eigenspaces, we get 
\begin{align*}
\Eig(\Sym^2(\Lambda^2V),1)= \Sym^2((\Lambda^2 V)_+)\oplus((\Lambda^2 V)_\xi \otimes (\Lambda^2 V)_{\overline{\xi}})
\end{align*}
 from which it directly follows that
\begin{align*}
    \eig(\Sym^2(\Lambda^2V),1)=
    \frac{27}{4}r^4 - \frac{189}{2}r^3 + \frac{1989}{4}r^2 - \frac{2331}{2}r + 18685.
\end{align*}
Putting everything together we obtain the following formula
\begin{align*}
   \eig(V_{(2,2)},1)&= \eig(\Sym^2(\Lambda^2V),1)- \eig(\Sym^2(V),1)- \eig(\Lambda^4V,1)=\\
   &=\frac{9}{2}r^4 - 60r^3 + \frac{597}{2}r^2 - 675r + 13158,
\end{align*}
and this finishes the proof.

\end{proof}

\begin{proof}(Theorem~\ref{MainThm})
    Let $\varphi$ be an involution on a manifold \(X\) of OG10 type. Then, by Proposition~\ref{prop:AB-decomposition} we find $\kg$-subrepresentations $A,B\subset H^*(X)$, such that $A\simeq \Ver, B\simeq \V22$ and $H^*(X) = A \oplus B$. By Proposition~\ref{prop:AB-eig} the action of $\varphi^*$ respects this decomposition, hence we have
    \[
    \eig(\varphi^*, H^*(X), 1) = \eig(\varphi^*, A, 1) + \eig(\varphi^*, B, 1).
    \]
    As explained in Proposition~\ref{prop: B-V22-compatibility}, the action of $\varphi^*$ on $B$ is compatible with the action of $\gamma_{(2,2)}$ on $\V22$. Then one gets either
    \[\eig(\varphi^*, B, 1) = \eig(\gamma_{(2,2)}, \V22,  1)\] or
    \begin{align*}
    \eig(\varphi^*, B, 1) = \eig(-\gamma_{(2,2)}, \V22, 1)
 &= \eig(\gamma_{(2,2)}, \V22, -1)\\
 &= \dim \V22-\eig(\gamma_{(2,2)}, \V22, 1)
    \end{align*}
    and the actual computation of the latter one in term of $r:=\dim H^2(X,\Q)^{\varphi^*}$ is an elementary exercise in representation theory, carried out in Lemma~\ref{lem:V22-compuation}.

The computation of $\eig(\varphi^*, A, 1)$ can be carried out similarly, or more directly using that $A\cap H^{2k}(X,\Q) = \Sym^k H^2(X,\Q)$ for $k\leq 5$ together with Poincaré duality and the well-known formula $\Sym^k(W_1\oplus W_2) = \bigoplus_{i+j = k} \Sym^i(W_1) \otimes \Sym^j(W_2)$.
One gets:
\begin{align*}
\eig(\varphi^*, A, 1) = \frac{2}{15}r^5 - \frac{22}{3}r^4 + 170r^3 - \frac{6182}{3}r^2 + \frac{194833}{15}r + 35702.
\end{align*}
From Proposition~\ref{prop: B-V22-compatibility} it follows that by summing to the latter the polynomial of Lemma~\ref{lem:V22-compuation} one gets
\begin{align*}
\eig(\varphi^*, H^*(X,\mathbb{Q}), 1)= \frac{2}{15}r^5 - \frac{20}{3}r^4 + \frac{422}{3}r^3 - \frac{4735}{3}r^2 + \frac{47381}{5}r + 64202,
\end{align*}
which is precisely the claim for an involution in case a).
In the same way, by subtracting the polynomial of Lemma~\ref{lem:V22-compuation} one gets
\begin{align*}
\eig(\varphi^*, H^*(X,\mathbb{Q}),1)=
\frac{2}{15}r^5 - 8r^4 + \frac{598}{3}r^3 -2543r^2 + \frac{247523}{15}r + 7202,
\end{align*}
which is the claim for an involution in case b).
\bigskip

The same argument can be applied when \(\varphi\) is an automorphism of order 3.
In this case we have a decomposition of \(H^{2}(X,\mathbb{C})\) in eigenspaces for the eigenvalues $1,\ \mathbb \zeta,\ \overline{\zeta}$. In particular, as observed above we have $\eig(H^2, \varphi^*, 1))= 2r$ for an integer $0\leq r\leq 12$.
The same computation as before shows that
\begin{align*}
    \eig(\sigma^*, A,1) = \frac{27}{20}r^5 - 18r^4 + \frac{357}{4}r^3 - 168r^2 - \frac{63}{5}r + 46674
\end{align*}
and summing this to the polynomial in Equation~\eqref{eq:ord3-V22} one gets 
\[\eig(\varphi^*,H^*(X,\mathbb{Q}),1))=\frac{27}{20}r^5 - \frac{27}{2}r^4 + \frac{117}{4}r^3 + \frac{261}{2}r^2 - \frac{3438}{5}r + 59832,\]
which concludes the proof.

\end{proof}

\begin{table}[] \caption{Possible polynomials \(f(r):=\eig(\varphi^*,H^*,1)\) in terms of $r$ defined as in the table.}
    \centering
    \begin{tabular}{ |c|c|c|c| }
 \hline
  & \(\ord(\varphi)\) & f(r)& r\\
  \hline
 a) & 2&\( \frac{2}{15}r^5 - \frac{20}{3}r^4 + \frac{422}{3}r^3 - \frac{4735}{3}r^2 + \frac{47381}{5}r + 64202.\)&\(r=\eig(\varphi^*,H^2(X,\bbQ),1)\)\\
 
 b) & 2&\(\frac{2}{15}r^5 - 8r^4 + \frac{598}{3}r^3 -2543r^2 + \frac{247523}{15}r + 7202 \)&\(r=\eig(\varphi^*,H^2(X,\bbQ),1)\)\\
 
 c) & 3& \( \frac{27}{20}r^5 - \frac{27}{2}r^4 + \frac{117}{4}r^3 + \frac{261}{2}r^2 - \frac{3438}{5}r + 59832\) &\(2r=\eig(\varphi^*,H^2(X,\bbQ),1)\)\\
 \hline
 \end{tabular}
    \label{table_pols}
\end{table}

\begin{proof}(Corollary~\ref{cor:no-Enriques})
Assume $X$ is a hyper-Kähler manifold of OG10 type, and $X\to Y$ an \'etale morphism given by taking the quotient by a finite group $G$ realizing $Y$ as an Enriques manifold.
    As observed in \cite{BN-WS}, if $d$ denotes the order of $G$, one has:
    \begin{align*}
        d \cdot \chi(Y,\mathcal O_Y) = \chi(X,\mathcal O_X) = \dim X/2 + 1 = 6.
    \end{align*}
    In order to prove the statement it suffices to consider groups $G=\langle \varphi\rangle$ of order 2 and 3 acting on $X$.
    
    If $\ord(\varphi)=2$, then 
    \begin{align*}
    176904 = \dim \mH^*(X,\bbQ) = e(X) = 2\cdot e(Y) = 2 \cdot \eig(\varphi^*, H^*(X), 1).
    \end{align*}
    The possible values of the right hand side are explicitly computed in Theorem~\ref{MainThm} (see Table \ref{table_pols}) in terms of $r:=\eig(\varphi^*, H^2(X), 1)$. One verifies that there does not exist an integer $0\leq r\leq 24=\dim H^2(X,\Q)$, that satisfies such an equality. This shows that no hyper-Kähler manifolds of OG10 type have \(2:1\) \'etale quotients.

 Reasoning in a similar way for an order $3$ automorphism $\varphi$, one gets 
 \[ 
 176904 = \dim \mH^*(X,\bbQ) = e(X) = 3 \cdot \dim \mH^*(Y, \mathbb Q) = 3 \cdot \eig(\varphi^*, H^*(X), 1).
 \]
The possible values of the latter are computed in terms of \(2r:=\eig(\varphi^*, H^2(X), 1)\) in Theorem~\ref{MainThm} (see Table \ref{table_pols}).
 As before, one gets no integer solution. This shows that no hyper-Kähler manifolds of OG10 type have \(3:1\) \'etale quotients.
\end{proof}

\begin{remark}
The decomposition \eqref{tot-LLV-dec} holds in slightly more general settings, as shown in \cite[Theorem 3.26]{GKLR}. Consequently, Corollary \ref{cor:no-Enriques} could be restated as follows:

There does not exist an Enriques manifold whose universal covering space is a hyper-Kähler tenfold \(X\) satisfying the following properties:
\begin{align*}
    &b_2(X) = 24, \\
    &e(X) = 176904, \\
    &H^{\mathrm{odd}}(X, \mathbb{Q}) = 0 \quad \text{(i.e., \(X\) has no cohomology in odd degree)}.
\end{align*}
\end{remark}
\bibliography{references}
\bibliographystyle{alpha}
\end{document}